\documentclass{amsart}
\usepackage[T1]{fontenc}
\usepackage[utf8]{inputenc}
\usepackage{amssymb,mathtools}
\usepackage{xurl}
\usepackage[hidelinks]{hyperref}
\usepackage{aliascnt}

\numberwithin{equation}{section}

\newtheorem{theorem}{Theorem}[section]
\newaliascnt{proposition}{theorem}
\newtheorem{proposition}[proposition]{Proposition}
\aliascntresetthe{proposition}
\newaliascnt{lemma}{theorem}
\newtheorem{lemma}[lemma]{Lemma}
\aliascntresetthe{lemma}
\newaliascnt{corollary}{theorem}

\aliascntresetthe{corollary}
\newaliascnt{claim}{theorem}

\aliascntresetthe{claim}
\theoremstyle{definition}
\newaliascnt{definition}{theorem}
\newtheorem{definition}[definition]{Definition}
\aliascntresetthe{definition}
\newaliascnt{remark}{theorem}

\aliascntresetthe{remark}
\newaliascnt{construction}{theorem}

\aliascntresetthe{construction}

\usepackage[nameinlink,noabbrev]{cleveref}
\crefname{theorem}{Theorem}{Theorems}
\Crefname{theorem}{Theorem}{Theorems}
\crefname{proposition}{Proposition}{Propositions}
\Crefname{proposition}{Proposition}{Propositions}
\crefname{lemma}{Lemma}{Lemmas}
\Crefname{lemma}{Lemma}{Lemmas}
\crefname{corollary}{Corollary}{Corollaries}
\Crefname{corollary}{Corollary}{Corollaries}
\crefname{definition}{Definition}{Definitions}
\Crefname{definition}{Definition}{Definitions}
\crefname{remark}{Remark}{Remarks}
\Crefname{remark}{Remark}{Remarks}
\crefname{claim}{Claim}{Claims}
\Crefname{claim}{Claim}{Claims}
\crefname{construction}{Construction}{Constructions}
\Crefname{construction}{Construction}{Constructions}
\crefname{section}{Section}{Sections}
\Crefname{section}{Section}{Sections}

\DeclareMathOperator{\conv}{conv}
\DeclareMathOperator{\supp}{supp}
\DeclareMathOperator{\Lip}{Lip}
\DeclareMathOperator{\sgn}{sgn}

\newcommand{\R}{\mathbb{R}}

\newcommand{\cL}{\mathcal{L}}
\newcommand{\cH}{\mathcal{H}}
\newcommand{\cK}{\mathcal{K}}
\newcommand{\cR}{\mathcal{R}}
\newcommand{\cA}{\mathcal{A}}
\newcommand{\cP}{\mathcal{P}}
\newcommand{\cI}{\mathcal{I}}

\newcommand{\norm}[1]{\left\lVert #1\right\rVert}
\newcommand{\matnorm}[1]{\left\lVert #1\right\rVert_{\mathrm F}}
\newcommand{\abs}[1]{\left\lvert #1\right\rvert}
\newcommand{\set}[1]{\left\{#1\right\}}

\newcommand{\dd}{\,\mathrm{d}}
\hypersetup{
  pdftitle={Rigidity of Unit Spheres in Finite-Dimensional Real Normed Spaces},
  pdfauthor={Jumpei Nakamura and Ryotaro Tanaka},
  pdfkeywords={Unit-sphere isometry, finite-dimensional normed space, Tingley's problem, Pl\"ucker contraction body, null Lagrangian}
}
\title[Rigidity of unit spheres]{Rigidity of Unit Spheres in\\
Finite-Dimensional Real Normed Spaces}

\author[J. Nakamura]{Jumpei Nakamura}
\address[J. Nakamura]{Department of Applied Mathematics, Graduate School of Science, Tokyo University of Science, Tokyo 162-8601, Japan}
\email{1425526@ed.tus.ac.jp}

\author[R. Tanaka]{Ryotaro Tanaka}
\address[R. Tanaka]{Katsushika Division, Institute of Arts and Sciences, Tokyo University of Science, Tokyo 125-8585, Japan}
\email{r-tanaka@rs.tus.ac.jp}

\thanks{The second author was supported by JSPS KAKENHI Grant Number JP24K06788.}
\date{}

\begin{document}

\begin{abstract}
We prove that finite-dimensional real normed spaces with isometric unit spheres are linearly isometric, where each sphere carries the distance induced by the ambient norm. This answers the object-level question of Kadets and Mart\'in in finite dimensions, without asserting linear extension of a prescribed sphere isometry. We associate with each norm a family of Pl\"ucker contraction bodies generated by volume-normalized maximal minors of linear contractions into finite-dimensional $\ell_\infty$ spaces. Null-Lagrangian identities show that the sphere metric determines these bodies. A finite exposed-face construction recovers almost-norming contractions from the convexified data, and compactness together with an exact volume identity yields a linear isometry.
\end{abstract}

\maketitle

\section{Introduction and main result}\label{sec:introduction}

The metric structure of a real normed space determines its affine structure: by the Mazur--Ulam theorem, every surjective isometry between real normed spaces is affine \cite{MazurUlam1932}. Mankiewicz's extension theorem gives a corresponding rigidity statement for bijective isometries between convex bodies with nonempty interior \cite{Mankiewicz1972}. It is natural to ask how much of the ambient space can be recovered from a smaller metric subspace. The unit sphere is a particularly stringent test. Its points record all directions, but its metric presentation does not include addition or scalar multiplication, and its empty interior prevents a direct application of convex-body extension theorems.

For a real normed space $X$, write
\[
  B_X=\set{x\in X:\norm{x}_X\leq1},
  \qquad
  S_X=\set{x\in X:\norm{x}_X=1}.
\]
Throughout the paper, the metric on the sphere is the restriction of the ambient norm distance,
\begin{equation}\label{eq:ambient-sphere-metric}
  d_X(x,y)=\norm{x-y}_X
  \qquad(x,y\in S_X),
\end{equation}
not an intrinsic or geodesic distance. A \emph{surjective sphere isometry} is a surjective map $\Delta:S_X\to S_Y$ satisfying
\begin{equation}\label{eq:sphere-isometry}
  \norm{\Delta(x)-\Delta(y)}_Y=\norm{x-y}_X
  \qquad(x,y\in S_X).
\end{equation}
Such a map is automatically injective. The question is whether the existence of this metric identification forces a linear isometric identification of $X$ and $Y$.

There are two distinct assertions here. Tingley's problem asks whether the \emph{prescribed map} $\Delta$ is the restriction of a surjective real-linear isometry of the ambient spaces \cite{Tingley1987}. We refer to this as the \emph{marked extension problem}. The \emph{object-level problem} asks only whether $X$ and $Y$ are linearly isometric whenever their unit spheres are isometric. Kadets and Mart\'in explicitly formulated this weaker question in \cite[Section~5]{KadetsMartin2012}, and emphasized its interest already in finite dimensions. The result of this paper answers the object-level question for arbitrary finite-dimensional real norms.

\begin{theorem}[Rigidity of unit spheres]\label{thm:main}
Let $X$ and $Y$ be finite-dimensional real normed spaces. If there exists a surjective sphere isometry $\Delta:S_X\to S_Y$, then there exists a surjective real-linear isometry $T:X\to Y$.
\end{theorem}

No smoothness, strict convexity, or polyhedrality is assumed, and equality of dimensions follows from the sphere isometry. Since a surjective linear isometry restricts to an isometry of unit spheres, the unit-sphere metric is a complete invariant of linear isometry among finite-dimensional real normed spaces. The isometry $T$ is not asserted to agree with the prescribed map $\Delta$ on $S_X$; Tingley's marked extension problem is not resolved here.

\subsection*{Relation to earlier work}

Tingley's original finite-dimensional result establishes preservation of antipodal points \cite{Tingley1987}. Subsequent work has obtained the stronger marked conclusion by exploiting particular geometric or algebraic structure. Kadets and Mart\'in proved linear extension for finite-dimensional polyhedral spaces \cite[Theorem~4.5]{KadetsMartin2012}. Banakh proved that every two-dimensional real Banach space has the Mazur--Ulam property: every surjective isometry from its unit sphere onto the unit sphere of a real Banach space extends to a surjective real-linear isometry \cite{Banakh2022}. This settles the marked problem in dimension two without a regularity assumption on the norm.

There are also object-level precedents within structured categories. The second author proved that two finite-dimensional $C^*$-algebras have isometric unit spheres if and only if they are $*$-isomorphic \cite[Corollary~2.5]{Tanaka2017}. This gives algebraic rigidity within that category. In a broader operator-algebraic setting, Mori and Ozawa established the Mazur--Ulam property for unital $C^*$-algebras \cite{MoriOzawa2020}. Peralta and \v{S}varc extended this result to unital JB$^*$-algebras \cite{PeraltaSvarc2025}. Peralta's survey describes the development of Tingley's problem for operator algebras \cite{Peralta2018}.

Hatori proved that every surjective isometry from the unit sphere of a uniform algebra onto that of an arbitrary complex Banach space extends to a surjective real-linear isometry \cite{Hatori2022}. Cabezas, Cueto-Avellaneda, Hirota, Miura, and Peralta established the same conclusion for commutative JB$^*$-triples \cite{CabezasEtAl2022}.

Other approaches extract geometric information about the unit ball from the sphere metric, including extremal and facial structure \cite{Tanaka2014,CamposJimenezGarciaPacheco2021}. A complementary local-to-global result of Cabello S\'anchez shows that, for finite-dimensional strictly convex spaces, a surjective sphere isometry that agrees with a linear map on a nonempty relatively open part of the sphere agrees with that map on the whole sphere \cite[Theorem~3.12]{CabelloSanchez2019}. The extension and classification results just discussed use low dimension or additional geometric, algebraic, or local-linearity hypotheses. Our argument applies to arbitrary finite-dimensional real norms and recovers the unit ball through a family of auxiliary convex bodies.

\subsection*{Proof architecture}

The proof separates the extraction of metric invariants from the recovery of a norm. In positive dimension, the radial extension
\[
  \widehat\Delta(0)=0,
  \qquad
  \widehat\Delta(x)=\norm{x}_X
    \Delta\!\left(\frac{x}{\norm{x}_X}\right)
  \quad(x\neq0)
\]
is a bi-Lipschitz homeomorphism from $X$ onto $Y$ and sends $B_X$ onto $B_Y$. It first allows us to recover the common dimension. It subsequently provides a comparison map for integral identities. After choosing linear coordinates, we work with two norms $p$ and $q$ on $E=\R^n$ and write $\delta:S_p\to S_q$ for the transported sphere isometry.

For a norm $r$ on $E$, let $v_r$ be the Lebesgue volume of $B_r$. If $A:E\to\ell_\infty^N$ is linear, with $N\geq n$, let $\mathbf m(A)$ denote the vector of its maximal minors, indexed by the increasing $n$-element subsets of $\{1,\ldots,N\}$. The associated \emph{Pl\"ucker contraction body} is
\[
  \cK_r^N
  =\conv\bigl\{\,\pm v_r\mathbf m(A):
        \norm{Ax}_\infty\leq r(x)\text{ for all }x\in E\,\bigr\}.
\]
It is a compact centrally symmetric convex set. The volume normalization is part of the invariant, and the signs allow comparison without selecting an orientation. Our analytic task, completed in \Cref{thm:body-equality}, is to prove
\[
  \cK_p^N=\cK_q^N\qquad(N\geq n);
\]
our convex-geometric task, completed in \Cref{thm:coordinate-rigidity}, is to recover a linear isometry from equality of this entire family.

To establish the first assertion, fix a $q$-contraction $A:E\to\ell_\infty^N$. Each coordinate of $A\circ\delta$ is $1$-Lipschitz on $S_p$, so McShane's extension theorem supplies a $1$-Lipschitz map $F:(E,p)\to\ell_\infty^N$ with that boundary trace \cite{McShane1934}. Rademacher's theorem implies that $DF(x)$ is a $p$-contraction almost everywhere \cite{Rademacher1919}. Consequently, the average of $v_p\mathbf m(DF)$ over $B_p$ belongs to $\cK_p^N$. The radial comparison map $A\circ\widehat\delta$ need not be a contraction, but it has the same boundary trace as $F$.

Maximal minors are classical null Lagrangians \cite{BallCurrieOlver1981}. We prove the boundary-trace identity needed here for Lipschitz maps on an arbitrary norm ball, and use it to replace $F$ by $A\circ\widehat\delta$ in the integral of each minor. A weak Piola argument then shows that the Jacobian sign of $\widehat\delta$ is constant almost everywhere. Signed change of variables identifies the resulting integral with $s_\delta v_q\mathbf m(A)$, where $s_\delta\in\{-1,1\}$. Central symmetry gives $\cK_q^N\subseteq\cK_p^N$, and the inverse sphere isometry gives the reverse inclusion. The Piola identity is classical; see \cite{KupfermanShachar2019} for a geometric treatment. We use these identities to transfer the contraction data between the two norms without linearizing $\delta$.

The recovery step must retain more information than membership in a convex hull. At each accuracy $0<\varepsilon<1$, we optimize a dual frame jointly with finitely many additional functionals, called satellite rows. The construction produces a supporting functional for which every raw maximizing configuration has a nonzero leading minor and almost norms the target sphere. An extreme point of the common supporting face belongs to both compact raw generating sets, by Milman's converse to the Krein--Milman theorem. A fixed nonzero leading minor permits Cramer's rule to compare the two representing matrices. This yields a bijective linear map $L_\varepsilon:E\to E$ such that
\[
  (1-\varepsilon)q(L_\varepsilon x)\leq p(x)
  \quad(x\in E),
  \qquad
  v_p\abs{\det L_\varepsilon}=v_q.
\]
The norm estimate is one-sided, but the determinant--volume identity is exact. Finite-dimensional compactness gives an invertible limit $L$ with $L(B_p)\subseteq B_q$ and equal volumes. Proper inclusion is then impossible for these convex bodies. Hence $L(B_p)=B_q$, which is the desired linear isometric equivalence. The proof thus combines classical analytic identities with a finite convex-recovery mechanism.

\subsection*{Relation to the preliminary report}
An earlier, non-peer-reviewed account of this work appeared as the AI-assisted \emph{Sphere Rigidity} Brief Report, Draft~R2 \cite{SphereRigidityBrief2026}. The present article builds on that report and provides a self-contained treatment with full proofs, expanded exposition, and a fuller comparison with earlier work.

\subsection*{Organization}
\Cref{sec:preliminaries} gives the preliminary reductions and coordinate models, and \Cref{sec:plucker-bodies} introduces the contraction bodies and their averaged derivative data. \Cref{sec:trace} proves boundary-trace invariance. \Cref{sec:orientation} establishes the orientation formula and uses it to identify the contraction bodies of the two norms. Finite recovery is developed in \Cref{sec:finite-recovery}; compactness and the proof of \Cref{thm:main} follow in \Cref{sec:compact-limit}. \Cref{sec:tingley} discusses the remaining marked problem and further questions.

\section{Preliminary reductions and coordinate models}\label{sec:preliminaries}

Let $\Delta:S_X\to S_Y$ be a surjective sphere isometry. Since $S_X=\varnothing$ if and only if $X=\{0\}$, either both spaces are zero, in which case the conclusion is immediate, or both are nonzero. We henceforth work in the latter case.

Define the radial extension by
\begin{equation}\label{eq:radial-extension}
  \widehat\Delta(0)=0,
  \qquad
  \widehat\Delta(x)=\norm{x}_X
    \Delta\!\left(\frac{x}{\norm{x}_X}\right)
  \quad(x\neq0).
\end{equation}
It preserves the norm of each vector:
\begin{equation}\label{eq:radial-norm}
  \norm{\widehat\Delta(x)}_Y=\norm{x}_X
  \qquad(x\in X).
\end{equation}

The bi-Lipschitz property of this extension is standard; see \cite[Section~5]{KadetsMartin2012}. We record the elementary estimate needed below.

\begin{proposition}[Bi-Lipschitz radial extension]\label{prop:radial-bilip}
The map $\widehat\Delta:X\to Y$ is a bijection with inverse $\widehat{\Delta^{-1}}$. Both maps are $3$-Lipschitz, and $\widehat\Delta(B_X)=B_Y$.
\end{proposition}

\begin{proof}
The inverse formula and the assertion about unit balls follow directly from \eqref{eq:radial-extension} and \eqref{eq:radial-norm}. For the Lipschitz estimate, write $x=ru$ and $y=sv$ with $r\geq s>0$ and $u,v\in S_X$. Then
\begin{align*}
  \norm{\widehat\Delta(x)-\widehat\Delta(y)}_Y
  &\leq r-s+s\norm{u-v}_X\\
  &\leq\norm{x-y}_X+2(r-s)
  \leq3\norm{x-y}_X.
\end{align*}
Here $s(u-v)=(x-y)-(r-s)u$ and $r-s\leq\norm{x-y}_X$. If one vector is zero, the estimate follows from \eqref{eq:radial-norm}. The same argument applies to $\Delta^{-1}$.
\end{proof}

The equality of dimensions is now a standard consequence of bi-Lipschitz invariance of Hausdorff dimension and equivalence of finite-dimensional norms:
\begin{equation}\label{eq:dimension-recovery}
  \dim_{\R}X=\dim_{\cH}B_X
  =\dim_{\cH}B_Y=\dim_{\R}Y.
\end{equation}
For the underlying Hausdorff-measure facts, see \cite[Sections~2.2 and~2.4.1]{EvansGariepy2015}. The same radial extension will serve as a comparison map in the integral identities below.

Set
\[
  n=\dim_{\R}X=\dim_{\R}Y\geq1,
  \qquad E=\R^n,
\]
and choose linear isomorphisms
\[
  e_X:E\longrightarrow X,
  \qquad
  e_Y:E\longrightarrow Y.
\]
Pull the two norms back to $E$ by setting
\begin{equation}\label{eq:model-norms}
  p(x)=\norm{e_Xx}_X,
  \qquad
  q(x)=\norm{e_Yx}_Y.
\end{equation}
For any norm $r$ on $E$, we use the notation
\[
  B_r=\set{x\in E:r(x)\leq1},
  \qquad
  S_r=\set{x\in E:r(x)=1},
  \qquad
  v_r=\cL^n(B_r),
\]
where $\cL^n$ denotes Lebesgue measure on $E$.

Conjugating $\Delta$ by $e_X$ and $e_Y$ gives a surjective map
\begin{equation}\label{eq:model-sphere-isometry}
  \delta:S_p\longrightarrow S_q
\end{equation}
satisfying
\[
  q(\delta(u)-\delta(v))=p(u-v)
  \qquad(u,v\in S_p).
\]

We equip $E$ with its standard Euclidean inner product and use the associated Lebesgue measure, Fr\'echet derivatives, distributional derivatives, and convolutions. Linear maps $E\to\R^N$ are represented by matrices whose rows are the coordinate functionals of the output. We write $\norm{\cdot}_2$ and $\norm{\cdot}_\infty$ for the Euclidean and supremum norms on coordinate spaces. Since all norms on a finite-dimensional space are equivalent, for each norm $r$ on $E$ there exist constants $0<\lambda_r\leq\Lambda_r$ such that
\begin{equation}\label{eq:model-comparison}
  \lambda_r\norm{x}_\infty\leq r(x)
  \leq\Lambda_r\norm{x}_\infty
  \qquad(x\in E).
\end{equation}
These choices affect only auxiliary estimates; all metric statements continue to be expressed in terms of $p$ and $q$.

We shall associate to each norm on $E$ a family of convex sets and show that this family determines its linear isometry class.

\section{Pl\"ucker contraction bodies}\label{sec:plucker-bodies}

Fix a norm $r$ on $E=\R^n$ and an integer $N\geq n$. We form a centrally symmetric convex set from the volume-normalized maximal minors of linear contractions $(E,r)\to\ell_\infty^N$, where $\ell_\infty^N$ denotes $\R^N$ with its supremum norm.

Maximal minors play two roles: their integrals depend only on boundary data (\Cref{thm:trace-invariance}), while their ratios recover row coordinates when a leading minor is nonzero (\Cref{lem:anchored-factorization}). The $\ell_\infty^N$ target allows scalar Lipschitz extensions to be combined without increasing the Lipschitz constant.

\subsection{Maximal-minor data}

Let $\cI(n,N)$ be the set of increasing $n$-element subsets of $\{1,\dots,N\}$, and let
\[
  \cP_{n,N}=\R^{\cI(n,N)}
\]
be the corresponding Pl\"ucker-coordinate space. If $A:E\to\R^N$ is linear, we write $A=(a_1,\ldots,a_N)$ for its ordered row functionals. For $I\in\cI(n,N)$, the map $A_I:E\to\R^n$ is obtained by retaining the rows indexed by $I$ in increasing order. All maximal-minor signs below refer to this fixed order.

\begin{definition}[Normalized Pl\"ucker vector]\label{def:normalized-plucker}
For a linear map $A:E\to\R^N$, define its maximal-minor vector by
\[
  \mathbf m(A)=\bigl(\det A_I\bigr)_{I\in\cI(n,N)}\in\cP_{n,N}.
\]
The normalized Pl\"ucker vector of $A$ relative to the norm $r$ is
\begin{equation}\label{eq:normalized-plucker}
  \Pi_r(A)=v_r\mathbf m(A).
\end{equation}
\end{definition}

\begin{definition}[Pl\"ucker contraction body]\label{def:plucker-body}
A linear map $A:E\to\ell_\infty^N$ is called an \emph{$r$-contraction} if
\[
  \norm{Ax}_\infty\leq r(x)
  \qquad(x\in E).
\]
The signed maximal-minor vectors before convexification form the \emph{raw Pl\"ucker set}; its convex hull is the \emph{Pl\"ucker contraction body}:
\begin{equation}\label{eq:raw-and-body}
  \cR_r^N
  =\set{\pm\Pi_r(A):A\text{ is an }r\text{-contraction}},
  \qquad
  \cK_r^N=\conv(\cR_r^N)\subset\cP_{n,N}.
\end{equation}
\end{definition}

The volume factor makes this construction invariant under linear changes of coordinates. Indeed, let $C\in GL(n,\R)$ and set $\widetilde r(x)=r(Cx)$. Then
\[
  v_{\widetilde r}=\frac{v_r}{\abs{\det C}},
  \qquad
  \mathbf m(AC)=(\det C)\mathbf m(A).
\]
Moreover, $A$ is an $r$-contraction if and only if $AC$ is a $\widetilde r$-contraction. Hence
\[
  \Pi_{\widetilde r}(AC)=\sgn(\det C)\Pi_r(A),
  \qquad
  \cK_{\widetilde r}^N=\cK_r^N.
\]
Thus volume normalization cancels the determinant magnitude, and the two signs remove its orientation. The same cancellation will give the exact volume identity in \Cref{lem:anchored-factorization}.

\begin{proposition}[Compactness]\label{prop:plucker-compact}
The set $\cR_r^N$ is nonempty, compact, and centrally symmetric. Its convex hull $\cK_r^N$ is compact and centrally symmetric, with nonempty interior in $\cP_{n,N}$.
\end{proposition}

\begin{proof}
The contraction matrices form the compact product of $N$ copies of the dual unit ball of $(E,r)$. Since $A\mapsto\Pi_r(A)$ is continuous, its image and the union of that image with its negative are compact. The convex hull of a compact set in a finite-dimensional space is compact by Carath\'eodory's theorem \cite[Theorem~3.1 and Corollary~3.1]{Gruber2007}. Nonemptiness and central symmetry follow from the definition.

For the interior assertion, fix an invertible matrix $B$ with rows in the dual unit ball and put $a=v_r\abs{\det B}>0$. Placing the rows of $B$ at the indices in $I$ and zero rows elsewhere gives an $r$-contraction whose only nonzero maximal minor is the $I$th one. Thus $\pm a e_I\in\cR_r^N$ for every $I\in\cI(n,N)$, where $(e_I)$ is the coordinate basis of $\cP_{n,N}$. The convex hull of these points has nonempty interior and is contained in $\cK_r^N$.
\end{proof}

\subsection{Boundary extensions and Pl\"ucker averages}

Fix a $q$-contraction $A:(E,q)\to\ell_\infty^N$. The map $g_A=A\circ\delta:S_p\to\ell_\infty^N$ is $1$-Lipschitz, since
\[
  \norm{g_A(u)-g_A(v)}_\infty
  \leq q(\delta(u)-\delta(v))=p(u-v).
\]
Write $g_A=(h_1,\ldots,h_N)$. By McShane's theorem \cite{McShane1934}, the coordinatewise formula
\begin{equation}\label{eq:boundary-extension}
  F_A=(F_{A,1},\ldots,F_{A,N}),
  \qquad
  F_{A,j}(x)=\inf_{u\in S_p}\bigl(h_j(u)+p(x-u)\bigr)
\end{equation}
defines a $1$-Lipschitz extension $(E,p)\to\ell_\infty^N$ of $g_A$. Rademacher's theorem \cite{Rademacher1919,EvansGariepy2015} gives differentiability almost everywhere, and the Lipschitz bound implies
\[
  \norm{DF_A(x)h}_\infty\leq p(h)
  \qquad(h\in E)
\]
at each differentiability point. Thus $DF_A(x)$ is a $p$-contraction almost everywhere.

\begin{definition}[Pl\"ucker average]\label{def:plucker-average}
Let $F:E\to\ell_\infty^N$ be Lipschitz. Its Pl\"ucker average over $B_r$ is
\begin{equation}\label{eq:plucker-average}
  \cA_r(F)
  =\frac{1}{v_r}\int_{B_r}\Pi_r(DF(x))\dd x
  =\int_{B_r}\mathbf m(DF(x))\dd x
  \in\cP_{n,N}.
\end{equation}
The integral is taken coordinatewise in the finite-dimensional space $\cP_{n,N}$. We use a measurable representative of the weak derivative of $F$, which agrees almost everywhere with the classical derivative, and assign the value $0$ on the remaining null set; this convention does not change the integral.
\end{definition}

The derivative of a Lipschitz map is measurable and essentially bounded, so its maximal minors are integrable on $B_r$. We shall use the following elementary fact about convex averages.

\begin{lemma}[Averages in closed convex sets]\label{lem:convex-average}
Let $K$ be a nonempty closed convex subset of a finite-dimensional real vector space, let $\mu$ be a probability measure, and let $h$ be an integrable map satisfying $h(\omega)\in K$ for $\mu$-almost every $\omega$. Then
\[
  \int h\,\dd\mu\in K.
\]
\end{lemma}

\begin{proof}
If the integral did not belong to $K$, the finite-dimensional separation theorem would provide a linear functional $\lambda$ and a number $\alpha$ such that $\lambda\leq\alpha$ on $K$ but $\lambda(\int h\,\dd\mu)>\alpha$. Since $h\in K$ almost everywhere,
\[
  \lambda\!\left(\int h\,\dd\mu\right)
  =\int\lambda(h)\,\dd\mu\leq\alpha,
\]
a contradiction.
\end{proof}

\begin{proposition}[The average lies in the contraction body]\label{prop:average-in-body}
If $F:E\to\ell_\infty^N$ is $1$-Lipschitz with respect to $r$, then
\[
  \cA_r(F)\in\cK_r^N.
\]
In particular, the extension $F_A$ from \eqref{eq:boundary-extension} satisfies
\[
  \cA_p(F_A)\in\cK_p^N.
\]
\end{proposition}

\begin{proof}
At almost every differentiability point, $DF(x)$ is an $r$-contraction, and therefore
\[
  \Pi_r(DF(x))\in\cR_r^N\subseteq\cK_r^N.
\]
The measure $v_r^{-1}\mathbf1_{B_r}\cL^n$ is a probability measure. Apply \Cref{lem:convex-average} to the closed convex set $\cK_r^N$.
\end{proof}

We next show that $\cA_p(F_A)$ depends only on the trace on $S_p$, so that it can be computed using the radial comparison map $A\circ\widehat\delta$.

\section{Boundary-trace invariance}\label{sec:trace}

The main conclusion of this section is the boundary-trace identity
\[
  F=G\text{ on }S_r\quad\Longrightarrow\quad\cA_r(F)=\cA_r(G)
\]
for Lipschitz maps $F,G:E\to\ell_\infty^N$ (\Cref{thm:trace-invariance}). Its key ingredient is compact-perturbation invariance (\Cref{prop:lipschitz-compact-perturb}), which is also used in \Cref{sec:orientation}. All analytic operations use the fixed Euclidean coordinates; norm equivalence makes the earlier maps Lipschitz in these metrics.

Fix $N\geq n$, set $Z=\R^N$, and let $I\in\cI(n,N)$. Denote by $P_I:Z\to\R^n$ the coordinate projection onto the components indexed by $I$. At points where a Lipschitz map $F:E\to Z$ is differentiable, define the selected Jacobian
\begin{equation}\label{eq:selected-jacobian}
  J_I F(x)=\det D(P_I\circ F)(x).
\end{equation}
Its value on the exceptional null set is irrelevant.

\subsection{Compactly supported perturbations}

The argument begins with the differential Piola identity. First-order null Lagrangians depending only on the derivative are affine combinations of minors, and their weak-continuity properties are characterized in \cite[Theorem~3.4 and the discussion on p.~136]{BallCurrieOlver1981}; a geometric treatment of the Piola identity is given in \cite{KupfermanShachar2019}. We nevertheless include the short Euclidean proof, because both the row convention and the compact-perturbation form are used explicitly below.

\begin{lemma}[Differential Piola identity]\label{lem:smooth-piola}
Let $g\in C^2(E;E)$, with the convention $(Dg)_{ij}=\partial_jg_i$. If $C_i(Dg)$ denotes the $i$th row of the cofactor matrix of $Dg$, then
\[
  \operatorname{div} C_i(Dg)=0
  \qquad(1\leq i\leq n).
\]
\end{lemma}

\begin{proof}
If $n=1$, then the cofactor matrix of $Dg$ is the constant matrix $(1)$, and the assertion is immediate. Assume that $n\geq2$ and fix $i$. With the convention that $\varepsilon_{j_1\ldots j_n}$ denotes the alternating symbol, the $j$th component of the $i$th cofactor row can be written as
\[
  C_{ij}(Dg)
  =\frac{1}{(n-1)!}
    \sum_{\substack{i_2,\ldots,i_n\\ j_2,\ldots,j_n}}
    \varepsilon_{i i_2\cdots i_n}
    \varepsilon_{j j_2\cdots j_n}
    \prod_{a=2}^n \partial_{j_a}g_{i_a}.
\]
When $\partial_j$ is applied and the result is summed over $j$, each term contains one mixed second derivative $\partial_j\partial_{j_a}g_{i_a}$. Interchanging the two column indices $j$ and $j_a$ leaves that second derivative unchanged and reverses the sign of the second alternating tensor. The terms therefore cancel in pairs. Hence
\[
  \sum_{j=1}^n\partial_j C_{ij}(Dg)=0.\qedhere
\]
\end{proof}

\begin{lemma}[Change of one output coordinate]\label{lem:one-output}
Let $g\in C^2(E;E)$, let $\phi\in C_c^1(E)$, and fix $i\in\{1,\dots,n\}$. Let $g^{(i,\phi)}$ be obtained from $g$ by replacing its $i$th component $g_i$ with $g_i+\phi$. Then
\[
  \int_E\bigl(\det Dg^{(i,\phi)}-\det Dg\bigr)\dd x=0.
\]
\end{lemma}

\begin{proof}
Expansion along the modified row and \Cref{lem:smooth-piola} give
\[
  \det Dg^{(i,\phi)}-\det Dg
  =\nabla\phi\cdot C_i(Dg)
  =\operatorname{div}\bigl(\phi C_i(Dg)\bigr).
\]
The vector field on the right is continuously differentiable and compactly supported, so its divergence has integral zero.
\end{proof}

\begin{proposition}[Smooth compact perturbations]\label{prop:smooth-compact-perturb}
Let $G,U:E\to Z$ be smooth, and suppose that $U$ has compact support. Then, for every $I\in\cI(n,N)$, the difference $J_I(G+U)-J_I(G)$ is compactly supported and integrable, and
\[
  \int_E\bigl(J_I(G+U)-J_I(G)\bigr)\dd x=0.
\]
\end{proposition}

\begin{proof}
Write $P_I\circ U=(u_1,\ldots,u_n)$ and set
\[
  H_0=P_I\circ G,
  \qquad
  H_k=H_{k-1}+u_k e_k
  \quad(1\leq k\leq n).
\]
Thus $H_n=P_I\circ(G+U)$. Each $u_k$ is smooth and compactly supported. Applying \Cref{lem:one-output} to the pair $(H_{k-1},u_k)$ gives
\[
  \int_E(\det DH_k-\det DH_{k-1})\dd x=0.
\]
The difference is supported in $\supp u_k$, and is therefore integrable. Summing over $k$ telescopes from $H_0$ to $H_n$ and proves the claim.
\end{proof}

We now pass from smooth to Lipschitz perturbations. Write $\matnorm{P}$ for the Frobenius norm of a matrix $P$. Since the determinant is a homogeneous polynomial of degree $n$, there exists a constant $C_n>0$ such that
\begin{equation}\label{eq:det-difference}
  \abs{\det P-\det Q}
  \leq C_n\matnorm{P-Q}
  \bigl(\matnorm{P}^{n-1}+\matnorm{Q}^{n-1}\bigr).
\end{equation}
Indeed, replace the columns of $Q$ by those of $P$ one at a time, expand the resulting telescoping sum by multilinearity, and bound each mixed determinant by a dimension-dependent product of Euclidean column norms.

\begin{lemma}[Strong $L^n$ continuity of the determinant]\label{lem:strong-ln-det}
Let $K\subset E$ be measurable with finite measure, with $n\geq1$. Let $P_k,P:K\to\R^{n\times n}$ be matrix fields. If $P_k\to P$ strongly in $L^n(K;\R^{n\times n})$, then
\[
  \det P_k\longrightarrow\det P
  \quad\text{strongly in }L^1(K).
\]
The same conclusion holds for any fixed $n\times n$ submatrix of rectangular matrix fields converging strongly in $L^n(K)$.
\end{lemma}

\begin{proof}
For $n=1$, the determinant is linear. Suppose that $n\geq2$. By \eqref{eq:det-difference} and H\"older's inequality,
\begin{align*}
  \norm{\det P_k-\det P}_{L^1(K)}
  &\leq C\norm{P_k-P}_{L^n(K)}\\
  &\quad\cdot
  \left(\norm{P_k}_{L^n(K)}^{n-1}
       +\norm{P}_{L^n(K)}^{n-1}\right).
\end{align*}
The second factor is bounded because $P_k\to P$ in $L^n(K)$, whereas the first tends to zero.
\end{proof}

\begin{proposition}[Lipschitz compact perturbations]\label{prop:lipschitz-compact-perturb}
Let $G,U:E\to Z$ be Lipschitz, and suppose that $U$ has compact support. Then, for every $I\in\cI(n,N)$, the difference $J_I(G+U)-J_I(G)$ is integrable, and
\[
  \int_E\bigl(J_I(G+U)-J_I(G)\bigr)\dd x=0.
\]
\end{proposition}

\begin{proof}
Choose a compact neighborhood $K$ of $\supp U$. For a standard mollifier $\rho_\eta$, set $G_\eta=G*\rho_\eta$ and $U_\eta=U*\rho_\eta$. For all sufficiently small $\eta>0$, one has $\supp U_\eta\subset K$, and \Cref{prop:smooth-compact-perturb} gives
\[
  \int_K\bigl(J_I(G_\eta+U_\eta)-J_I(G_\eta)\bigr)\dd x=0.
\]
Since Lipschitz maps belong locally to $W^{1,\infty}$, mollifier approximation \cite[Section~4.2.1]{EvansGariepy2015} yields
\[
  DG_\eta\to DG,\qquad DU_\eta\to DU
  \quad\text{strongly in }L^n(K).
\]
By \Cref{lem:strong-ln-det},
\[
  J_I(G_\eta+U_\eta)-J_I(G_\eta)
  \longrightarrow J_I(G+U)-J_I(G)
  \quad\text{in }L^1(K).
\]
All these differences vanish almost everywhere outside $K$. Passing to the limit proves the assertion; no integrability of the individual whole-space determinants is required.
\end{proof}

\subsection{Patching across a norm sphere}

Let $r$ be a norm on $E$. To apply the compact-perturbation result, we convert equality on $S_r$ into a globally Lipschitz perturbation supported in $B_r$.

\begin{lemma}[Zero extension]\label{lem:zero-extension}
Let $u:E\to Z$ be Lipschitz and satisfy $u=0$ on $S_r$. Define
\[
  u_0(x)=
  \begin{cases}
    u(x),&x\in B_r,\\
    0,&x\notin B_r.
  \end{cases}
\]
Then $u_0$ is globally Lipschitz and has compact support.
\end{lemma}

\begin{proof}
Only pairs of points on opposite sides of $S_r$ require attention. Let $x\in B_r$ and $y\notin B_r$. By continuity of $r$ along the segment and convexity of $B_r$, the line segment from $x$ to $y$ meets $S_r$ at a point $z$. Since $u(z)=0$ and $z$ lies between $x$ and $y$,
\[
  \norm{u_0(x)-u_0(y)}_2
  =\norm{u(x)-u(z)}_2
  \leq\Lip(u)\norm{x-z}_2
  \leq\Lip(u)\norm{x-y}_2.
\]
The remaining cases follow directly from the Lipschitz property of $u$. Finally, $\supp u_0\subseteq B_r$, and $B_r$ is compact.
\end{proof}

\begin{lemma}[The unit sphere is null]\label{lem:sphere-null}
If $n\geq1$, then $\cL^n(S_r)=0$.
\end{lemma}

\begin{proof}
For $0<t<1$, the inclusion $S_r\subset B_r\setminus tB_r$ gives
\[
  0\leq\cL^n(S_r)\leq(1-t^n)v_r\longrightarrow0
  \qquad(t\uparrow1).\qedhere
\]
\end{proof}

\begin{theorem}[Boundary-trace invariance]\label{thm:trace-invariance}
Let $r$ be a norm on $E=\R^n$ with $n\geq1$, and let $F,G:E\to\ell_\infty^N$ be Lipschitz maps satisfying
\[
  F=G\quad\text{on }S_r.
\]
Then, for every $I\in\cI(n,N)$,
\begin{equation}\label{eq:minor-trace-invariance}
  \int_{B_r}J_I F(x)\dd x
  =\int_{B_r}J_I G(x)\dd x.
\end{equation}
Equivalently,
\[
  \cA_r(F)=\cA_r(G).
\]
\end{theorem}

\begin{proof}
Set $u=F-G$, and let $u_0$ be the zero extension of $u|_{B_r}$ provided by \Cref{lem:zero-extension}. Define $H=G+u_0$. Then $H=F$ on $B_r$, $H=G$ on $E\setminus B_r$, and $H-G$ is a compactly supported Lipschitz map. By \Cref{prop:lipschitz-compact-perturb},
\[
  \int_E\bigl(J_I H-J_I G\bigr)\dd x=0.
\]
The integrand vanishes almost everywhere outside $B_r$ and equals $J_I F-J_I G$ on the interior of $B_r$. Since $S_r=\partial B_r$ is null by \Cref{lem:sphere-null}, this proves \eqref{eq:minor-trace-invariance}. Applying the identity to every maximal minor gives equality of the Pl\"ucker averages.
\end{proof}

\Cref{thm:trace-invariance} is the only boundary-trace result used later. It allows the Pl\"ucker average to be computed from any Lipschitz extension of the same sphere data. The next section evaluates the average for the radial extension and combines the two identities to prove metric invariance of the contraction bodies.
\section{Orientation and metric invariance}\label{sec:orientation}

The constancy of the Jacobian sign is classical for bi-Lipschitz homeomorphisms of connected domains; see \cite[Introduction and Section~2]{HenclMaly2010} for its relation to topological orientation. We give an analytic proof using \Cref{prop:lipschitz-compact-perturb}, so that the signed integral formula follows from the same identity as the boundary-trace theorem.

Let
\[
  f=\widehat\delta:E\longrightarrow E
\]
be the radial extension of the sphere isometry $\delta:S_p\to S_q$, and let $g=\widehat{\delta^{-1}}$ be its inverse. Our goal is the identity
\[
  \int_{B_p}\det Df(x)\dd x=s_\delta v_q
  \qquad\text{for some }s_\delta\in\{-1,1\},
\]
which also determines the averaged minors of $A\circ f$ (\Cref{thm:orientation-sign}). We show that $Df$ is nondegenerate almost everywhere and that its sign, transported to the target, has zero weak gradient. Set
\[
  U=\operatorname{int}B_p,
  \qquad
  V=\operatorname{int}B_q.
\]
By \Cref{prop:radial-bilip}, the restrictions $f|_U:U\to V$ and $g|_V:V\to U$ are inverse bi-Lipschitz homeomorphisms. Combining the lower Lipschitz estimate for $f$ with equivalence of the Euclidean and norm metrics, we obtain a constant $a>0$ such that
\begin{equation}\label{eq:anti-lipschitz}
  a\norm{x-x'}_2\leq\norm{f(x)-f(x')}_2
  \qquad(x,x'\in E).
\end{equation}
Choose a Borel measurable representative of $Df$ that agrees almost everywhere with the classical derivative, assigning the value $0$ on a Borel null set containing the exceptional points, and write
\[
  J_f(x)=\det Df(x).
\]
The function $J_f$ is measurable and essentially bounded on bounded sets.

\subsection{Nondegeneracy and signed change of variables}

\begin{lemma}[Almost-everywhere nondegeneracy]\label{lem:derivative-nondegenerate}
At almost every point $x\in U$, the derivative $Df(x)$ exists and is invertible. In particular, $J_f(x)\neq0$ for almost every $x\in U$.
\end{lemma}

\begin{proof}
Rademacher's theorem gives differentiability almost everywhere. Let $x$ be a differentiability point outside the chosen exceptional null set. Applying \eqref{eq:anti-lipschitz} to $x$ and $x+tv$, dividing by $\abs t$, and letting $t\to0$, we obtain
\[
  a\norm{v}_2\leq\norm{Df(x)v}_2
  \qquad(v\in E).
\]
Thus $Df(x)$ is injective. Its domain and codomain both have dimension $n$, so it is invertible.
\end{proof}

Define the source sign by
\[
  \varepsilon(x)=
  \begin{cases}
    1,&J_f(x)\geq0,\\
    -1,&J_f(x)<0.
  \end{cases}
\]
Then $\varepsilon$ is Borel measurable, takes values in $\{-1,1\}$, and satisfies
\[
  J_f(x)=\varepsilon(x)\abs{J_f(x)}
\]
for every $x$. By \Cref{lem:derivative-nondegenerate}, $J_f$ is nonzero for almost every $x\in U$. Transport it to the target by setting
\begin{equation}\label{eq:target-sign}
  \sigma(y)=\varepsilon(g(y))
  \qquad(y\in V).
\end{equation}
The inverse map $g$ is continuous, so $\sigma$ is Borel measurable and bounded by $1$ in absolute value. The values assigned on exceptional null sets will not affect any integral.

Both $f$ and $g$ send Lebesgue-null sets to null sets, as do all Lipschitz maps between subsets of $\R^n$ \cite[Sections~2.2 and~2.4.1]{EvansGariepy2015}. In particular, composition with $f$ preserves measurability of Lebesgue representatives and does not depend, modulo null sets, on the chosen representative.

\begin{lemma}[Signed change of variables]\label{lem:signed-transfer}
Let $\varphi:V\to\R$ be a measurable representative of an element of $L^1(V)$. Then both sides below are absolutely integrable and
\begin{equation}\label{eq:signed-transfer}
  \int_U\varphi(f(x))J_f(x)\dd x
  =\int_V\varphi(y)\sigma(y)\dd y.
\end{equation}
The value of either side depends only on the $L^1$ class of $\varphi$.
\end{lemma}

\begin{proof}
The weighted area formula for the injective Lipschitz map $f|_U:U\to V$ gives
\begin{equation}\label{eq:absolute-area-transfer}
  \int_U\abs{\varphi(f(x))}\,\abs{J_f(x)}\dd x
  =\int_V\abs{\varphi(y)}\dd y<\infty;
\end{equation}
see \cite[Theorem~3.9]{EvansGariepy2015} or \cite[3.2.3(1)]{Federer1969}. Since $\sigma\circ f=\varepsilon$ and $J_f=\varepsilon\abs{J_f}$, applying the same formula to $\varphi\sigma\in L^1(V)$ yields
\[
  \int_U\varphi(f(x))J_f(x)\dd x
  =\int_U(\varphi\sigma)(f(x))\abs{J_f(x)}\dd x
  =\int_V\varphi(y)\sigma(y)\dd y.
\]
Equation \eqref{eq:absolute-area-transfer} also proves independence of the representative.
\end{proof}

\subsection{Weak Piola identity and metric invariance}

\begin{lemma}[Weak Piola identity]\label{lem:weak-piola}
For every vector field $W\in C_c^1(V;\R^n)$,
\begin{equation}\label{eq:weak-piola}
  \int_U (\operatorname{div}W)(f(x))J_f(x)\dd x=0.
\end{equation}
\end{lemma}

\begin{proof}
Extend $W$ by zero to a field in $C_c^1(E;E)$. Write $e_i$ for the $i$th standard basis vector and $DW_i$ for the row derivative of the scalar component $W_i$. For each $i$, the map
\[
  u_i(x)=W_i(f(x))e_i
\]
is Lipschitz, with $\supp u_i\subseteq g(\supp W)\Subset U$. At almost every differentiability point of $f$, the chain rule gives
\[
  D(f+u_i)(x)
  =\bigl(I+e_iDW_i(f(x))\bigr)Df(x).
\]
Only row $i$ of $e_iDW_i$ can be nonzero, so
\begin{align*}
  \det\bigl(I+e_iDW_i(f(x))\bigr)
    &=1+\partial_iW_i(f(x)),\\
  \det D(f+u_i)(x)-\det Df(x)
    &=\partial_iW_i(f(x))J_f(x).
\end{align*}
By \Cref{prop:lipschitz-compact-perturb}, the determinant difference has integral zero. It vanishes almost everywhere outside $U$, hence
\[
  \int_U\partial_iW_i(f(x))J_f(x)\dd x=0.
\]
Summing over $i$ proves \eqref{eq:weak-piola}.
\end{proof}

Applying the signed change-of-variables formula to the weak Piola identity gives
\begin{equation}\label{eq:target-sign-weak-divergence}
  \int_V\sigma(y)\operatorname{div}W(y)\dd y=0
  \qquad(W\in C_c^1(V;\R^n)).
\end{equation}
To interpret this identity, recall that the distributional derivative of a locally integrable function $h$ on an open set $\Omega$ acts on smooth, compactly supported test functions by
\[
  \langle\partial_i h,\varphi\rangle
  =-\int_\Omega h\,\partial_i\varphi\dd x
  \qquad(\varphi\in C_c^\infty(\Omega)).
\]
For smooth $h$, this is the usual integration-by-parts formula; the right-hand side, however, requires only local integrability. Thus $h$ has zero distributional gradient precisely when these integrals vanish for every $i$ and every test function. This is the weak-derivative convention of \cite[Section~4.1]{EvansGariepy2015}. Taking $W=\varphi e_i$ in \eqref{eq:target-sign-weak-divergence} gives exactly this condition for $\sigma$. We recall why it implies constancy.

\begin{lemma}[Functions with vanishing weak gradient]\label{lem:weak-gradient-constant}
Let $\Omega\subset\R^n$ be connected and open, and let $h\in L^1_{\mathrm{loc}}(\Omega)$ satisfy
\[
  \int_\Omega h\,\partial_i\varphi\dd x=0
  \qquad(\varphi\in C_c^\infty(\Omega),\ 1\leq i\leq n).
\]
Then $h$ is almost everywhere equal to a constant on $\Omega$.
\end{lemma}

\begin{proof}
Fix an open Euclidean ball $B\Subset\Omega$. Let $\rho_\eta$ be a nonnegative smooth mollifier of integral one, supported in the ball of radius $\eta$ about zero, and set
\[
  h_\eta(x)=\int_\Omega h(y)\rho_\eta(x-y)\dd y.
\]
For sufficiently small $\eta>0$ and $x\in B$, the function $y\mapsto\rho_\eta(x-y)$ belongs to $C_c^\infty(\Omega)$. Differentiating the convolution and applying the hypothesis to this test function gives
\[
  \begin{aligned}
    \partial_i h_\eta(x)
    &=\int_\Omega h(y)\,\partial_{x_i}\rho_\eta(x-y)\dd y\\
    &=-\int_\Omega h(y)\,\partial_{y_i}\rho_\eta(x-y)\dd y=0.
  \end{aligned}
\]
Hence the smooth function $h_\eta$ is constant on $B$, say $h_\eta=c_\eta$ there.

By mollifier convergence \cite[Theorem~4.1]{EvansGariepy2015}, $h_\eta\to h$ in $L^1(B)$. Writing $|B|=\cL^n(B)$, we have
\[
  c_\eta=\frac{1}{|B|}\int_B h_\eta\dd x
  \longrightarrow\frac{1}{|B|}\int_B h\dd x=:h_B.
\]
Thus $h=h_B$ almost everywhere on $B$.

If two such balls overlap, their intersection has positive measure, so their almost-everywhere constants agree. Connectedness joins any two of these balls by a finite chain of overlapping balls compactly contained in $\Omega$. A countable cover by such balls then gives one constant outside a single null set in $\Omega$.
\end{proof}

Because $V$ is convex, it is connected. Since $\abs{\sigma}=1$ almost everywhere, \eqref{eq:target-sign-weak-divergence} and \Cref{lem:weak-gradient-constant} yield a sign
\[
  s_\delta\in\{-1,1\}
\]
such that $\sigma=s_\delta$ almost everywhere on $V$. We may now integrate the Jacobian with its fixed sign and, by the chain rule, compute every averaged maximal minor of a linear map composed with $f$.

\begin{theorem}[Orientation of the radial extension]\label{thm:orientation-sign}
There exists $s_\delta\in\{-1,1\}$ such that
\begin{equation}\label{eq:radial-signed-volume}
  \int_{B_p}\det Df(x)\dd x=s_\delta v_q.
\end{equation}
Moreover, for every integer $N\geq n$ and every linear map $A:E\to\ell_\infty^N$,
\begin{equation}\label{eq:radial-plucker-average}
  \cA_p(A\circ f)=s_\delta\Pi_q(A).
\end{equation}
\end{theorem}

\begin{proof}
The preceding argument shows that the transported sign is almost everywhere equal to $s_\delta$. Since $V$ has finite measure, \Cref{lem:signed-transfer} applied to the constant function $1$ gives
\[
  \int_U J_f(x)\dd x
  =s_\delta\cL^n(V)
  =s_\delta v_q.
\]
By \Cref{lem:sphere-null}, the boundary of $B_p$ is null, so the same identity holds over the closed ball $B_p$.

Fix $I\in\cI(n,N)$. At differentiability points of $f$, the chain rule and multiplicativity of the determinant give
\[
  \det D(A_I\circ f)(x)=\det(A_I)\det Df(x).
\]
Integrating and using \eqref{eq:radial-signed-volume} yields
\[
  \int_{B_p}\det D(A_I\circ f)(x)\dd x
  =s_\delta v_q\det(A_I),
\]
which is the $I$th coordinate of \eqref{eq:radial-plucker-average}. Since $I$ was arbitrary, the vector identity follows.
\end{proof}

The sign $s_\delta$ depends on the chosen orientations. Central symmetry removes it when we combine the signed-volume formula with boundary-trace invariance, completing the analytic part of the proof.

\begin{theorem}[Pl\"ucker-body invariance]\label{thm:body-equality}
For the coordinate norms $p$ and $q$ associated with the sphere isometry,
\begin{equation}\label{eq:all-bodies-equal}
  \cK_p^N=\cK_q^N
  \qquad(N\geq n).
\end{equation}
\end{theorem}

\begin{proof}
Fix a $q$-contraction $A:E\to\ell_\infty^N$. The McShane extension $F_A$ and the Lipschitz map $A\circ\widehat\delta$ agree on $S_p$. By \Cref{thm:trace-invariance,thm:orientation-sign,prop:average-in-body},
\[
  s_\delta\Pi_q(A)
  =\cA_p(A\circ\widehat\delta)
  =\cA_p(F_A)\in\cK_p^N.
\]
Central symmetry gives $\pm\Pi_q(A)\in\cK_p^N$. Taking convex hulls, we obtain $\cK_q^N\subseteq\cK_p^N$. The inverse sphere isometry gives the reverse inclusion.
\end{proof}

\section{Finite recovery from the contraction bodies}\label{sec:finite-recovery}

We now recover a linear isometry from the equality of contraction bodies established in \Cref{thm:body-equality}. The argument uses no further information about the original sphere isometry. Its finite step produces an almost-contractive linear map with an exact determinant--volume identity; the next section passes to an isometric limit.

\subsection{Almost maximal dual frames}

For a norm $r$ on $E$, let
\[
  B_{r^*}=\set{b\in E^*: \abs{b(x)}\leq r(x)
    \text{ for every }x\in E}
\]
be the closed unit ball of the dual norm. Equivalently,
\begin{equation}\label{eq:dual-representation}
  r(x)=\max_{b\in B_{r^*}}\abs{b(x)}
  \qquad(x\in E).
\end{equation}
This is the usual dual representation furnished by the Hahn--Banach theorem; see \cite[Corollary~III.6.5]{Conway1990}.

An ordered $n$-tuple $B=(b_1,\dots,b_n)\in(B_{r^*})^n$ will be called an \emph{$r$-admissible row system}. We also write $B$ for the corresponding $n\times n$ matrix. If $s\in E^*$ and $1\leq i\leq n$, then $B[i\leftarrow s]$ denotes the matrix obtained from $B$ by replacing its $i$th row by $s$.

Define
\begin{equation}\label{eq:det-max}
  d_r=\max\set{\abs{\det B}:B\in(B_{r^*})^n}.
\end{equation}
The maximum exists because $(B_{r^*})^n$ is compact. It is positive: by \eqref{eq:model-comparison}, the rows $\lambda_r e_1^*,\dots,\lambda_r e_n^*$ belong to $B_{r^*}$, and their determinant is $\lambda_r^n>0$.

Maximizing $\abs{\det B}$ is the classical volume method for constructing Auerbach bases \cite{Taylor1947}. We need the row-replacement estimate uniformly for frames whose determinant is close to $d_r$.

\begin{lemma}[Uniform inverse bound near the maximum]\label{lem:near-max-inverse}
Fix $0<\eta<d_r$. There is a constant $K=K(r,\eta)\geq1$ such that, whenever $B$ is an $r$-admissible row system satisfying
\[
  \abs{\det B}\geq d_r-\eta,
\]
then $B$ is invertible and
\begin{equation}\label{eq:inverse-bound}
  r(B^{-1}c)\leq K\norm{c}_\infty
  \qquad(c\in\R^n).
\end{equation}
\end{lemma}

\begin{proof}
The lower bound on $\abs{\det B}$ makes $B$ invertible. For $s\in B_{r^*}$, multilinearity in the replaced row gives
\begin{equation}\label{eq:cramer-row}
  \sum_{i=1}^n c_i\det B[i\leftarrow s]
  =(\det B)\,s(B^{-1}c).
\end{equation}
Indeed, writing $s=\alpha B$ gives $\det B[i\leftarrow s]=\alpha_i\det B$. Each replacement matrix is $r$-admissible, so
\[
  \abs{s(B^{-1}c)}
  \leq\frac{n d_r}{d_r-\eta}\norm{c}_\infty.
\]
Taking the maximum in \eqref{eq:dual-representation} proves \eqref{eq:inverse-bound} with $K=n d_r/(d_r-\eta)\geq1$.
\end{proof}

\begin{lemma}[Finite coefficient net]\label{lem:finite-coefficient-net}
Fix $0<\eta<d_r$, let $K\geq1$ be as in \Cref{lem:near-max-inverse}, and let $\rho>0$. There exists a finite set
\[
  C\subset\set{c\in\R^n:K^{-1}\leq\norm{c}_\infty\leq1}
\]
with the following property: for every $r$-admissible row system $B$ satisfying $\abs{\det B}\geq d_r-\eta$ and every $x\in S_r$, there exists $c\in C$ such that
\begin{equation}\label{eq:coefficient-approximation}
  r(x-B^{-1}c)\leq K\rho.
\end{equation}
\end{lemma}

\begin{proof}
Choose a finite $\rho$-net $C$ in the compact annulus
\[
  \set{c\in\R^n:K^{-1}\leq\norm{c}_\infty\leq1}.
\]
For an admissible $B$ as in the statement and $x\in S_r$, contractivity and \eqref{eq:inverse-bound} give
\[
  1=r(x)=r(B^{-1}Bx)\leq K\norm{Bx}_\infty,
  \qquad \norm{Bx}_\infty\leq1.
\]
Thus $Bx$ belongs to this annulus. Choose $c\in C$ with $\norm{Bx-c}_\infty\leq\rho$ and apply \eqref{eq:inverse-bound} to obtain
\[
  r(x-B^{-1}c)=r\bigl(B^{-1}(Bx-c)\bigr)\leq K\rho.\qedhere
\]
\end{proof}

\subsection{Satellite rows and an almost-norming exposed face}

A single almost-norming family is not enough: the common raw point selected later may be represented by a different family. We therefore seek a supporting functional for which \emph{every raw maximizing configuration} has an invertible leading block and almost norms $S_r$. We fix the coefficient net before jointly optimizing the leading block $B$ and the satellite rows $s_c$. Maximality will force $\abs{s_c(B^{-1}c)}=r(B^{-1}c)$ for every $c$; the uniform conclusion in (iv) below is what is needed at the common raw point.

\begin{proposition}[An exposed face of almost-norming configurations]\label{prop:satellite-package}
Let $r$ be a norm on $E$ and let $0<\varepsilon<1$. Then there exist an integer $N\geq n$, an $r$-contraction
\[
  A:E\longrightarrow\ell_\infty^N,
\]
and a linear functional $\ell\in\cP_{n,N}^*$ such that:
\begin{enumerate}
  \item[(i)] $\norm{Ax}_\infty>1-\varepsilon$ for every $x\in S_r$;
  \item[(ii)] the first $n$ rows of $A$ form an invertible matrix;
  \item[(iii)] $\Pi_r(A)$ belongs to the raw exposed slice
  \[
    \set{z\in\cR_r^N:
      \ell(z)=\max_{w\in\cK_r^N}\ell(w)};
  \]
  \item[(iv)] for every $r$-contraction $A':E\to\ell_\infty^N$ and every sign $\tau\in\{-1,1\}$ such that
  \[
    \ell(\tau\Pi_r(A'))=\max_{z\in\cK_r^N}\ell(z),
  \]
  the map $A'$ satisfies \emph{(i)} and \emph{(ii)}.
\end{enumerate}
\end{proposition}

\begin{proof}
The parameters are fixed in the order
\[
  \eta\longrightarrow K\longrightarrow\rho\longrightarrow C
  \longrightarrow w\longrightarrow\text{a maximizer of }\abs{P_w}.
\]
In particular, the coefficient net does not depend on the maximizing configuration.

\emph{Step 1: a maximizing configuration and its leading determinant.}
Set $\eta=d_r/2$. The explicit estimate in \Cref{lem:near-max-inverse} permits $K=2n$. Choose $\rho=\varepsilon/(8n)$, so that
\begin{equation}\label{eq:satellite-radius}
  2K\rho=\varepsilon/2<\varepsilon.
\end{equation}
Fix a coefficient net $C=\{c^{(1)},\ldots,c^{(m)}\}$ from \Cref{lem:finite-coefficient-net}, and set $w=2mn+1$. A configuration consists of an admissible base matrix $B$ and satellite rows $s_c\in B_{r^*}$, one for each $c\in C$. On the compact configuration space $(B_{r^*})^{n+m}$, consider
\begin{equation}\label{eq:satellite-poly}
  P_w\bigl(B,(s_c)_{c\in C}\bigr)
  =w\det B+\sum_{c\in C}\sum_{i=1}^n
     c_i\det B[i\leftarrow s_c].
\end{equation}
The satellite sum has absolute value at most $M_C=mnd_r$, since $\norm{c}_\infty\leq1$ and every replacement determinant has absolute value at most $d_r$. Our choice of $w$ gives
\begin{equation}\label{eq:satellite-gap}
  M_C=mnd_r<(2mn+1)d_r/2=w\eta.
\end{equation}
Let $(B,(s_c))$ maximize $\abs{P_w}$, with maximum $M$. A determinant-maximizing base with all satellites zero gives $M\geq wd_r$, whence
\begin{equation}\label{eq:maximizer-determinant-gap}
  wd_r\leq M\leq w\abs{\det B}+M_C,
  \qquad
  \abs{\det B}\geq d_r-M_C/w>d_r-\eta>0.
\end{equation}
Thus the same inverse bound $K$ applies to the leading block of every maximizing configuration.

\emph{Step 2: maximality forces each satellite row to be norming.}
Fix $c\in C$, put $x_c=B^{-1}c$ and $D=\det B$, and hold all rows other than $s_c$ fixed. By \eqref{eq:cramer-row}, the resulting polynomial is $R+D s_c(x_c)$ for a constant $R$. The dual ball is symmetric and contains a norming functional for $x_c$, so maximality gives
\[
  M=\max_{s\in B_{r^*}}\abs{R+D s(x_c)}
   =\abs R+\abs D\,r(x_c).
\]
The chosen row $s_c$ attains this maximum. Since $D\neq0$, the triangle inequality forces
\[
  \abs{s_c(B^{-1}c)}=r(B^{-1}c)
  \qquad(c\in C).
\]

\emph{Step 3: the stacked contraction almost norms the sphere.}
Let $N=n+m$ and stack the base rows followed by the satellites to form $A:E\to\ell_\infty^N$. All rows lie in $B_{r^*}$, so $A$ is an $r$-contraction. For $x\in S_r$, choose $c\in C$ with $r(x-x_c)\leq K\rho$, as in \Cref{lem:finite-coefficient-net}. Then
\begin{equation}\label{eq:satellite-almost-norming}
  \begin{aligned}
    \norm{Ax}_\infty
    &\geq\abs{s_c(x)}
     \geq r(x_c)-r(x-x_c)\\
    &\geq1-2r(x-x_c)
     \geq1-2K\rho>1-\varepsilon.
  \end{aligned}
\end{equation}
This proves (i); (ii) follows from \eqref{eq:maximizer-determinant-gap}.

\emph{Step 4: encoding the polynomial as a Pl\"ucker functional.}
Let $I_0=\{1,\ldots,n\}$ and $I_{\nu,i}=I_0\setminus\{i\}\cup\{n+\nu\}$. For any configuration $A'$ with leading block $B'$ and satellites $s'_{c^{(\nu)}}$, the fixed increasing row order gives
\[
  \det A'_{I_{\nu,i}}
  =(-1)^{n-i}\det B'[i\leftarrow s'_{c^{(\nu)}}].
\]
Consequently,
\[
  \ell_0(z)=wz_{I_0}
     +\sum_{\nu=1}^m\sum_{i=1}^n
       (-1)^{n-i}c_i^{(\nu)}z_{I_{\nu,i}}
\]
satisfies $\ell_0(\Pi_r(A'))=v_rP_w(A')$. Since $M\geq wd_r>0$, we may set $\ell=\sgn(P_w(A))\ell_0$. For every raw point $\tau\Pi_r(A')$, with $\tau\in\{-1,1\}$,
\begin{equation}\label{eq:raw-maximality}
  \ell(\tau\Pi_r(A'))\leq v_r\abs{P_w(A')}
    \leq v_rM=\ell(\Pi_r(A)).
\end{equation}
The maximum of a linear functional is unchanged on passing to the convex hull. Thus $\Pi_r(A)$ maximizes $\ell$ on $\cK_r^N$, proving (iii).

\emph{Step 5: verification of (i) and (ii) for every raw maximizer.}
Let $z=\tau\Pi_r(A')$ satisfy
$\ell(z)=\max_{\cK_r^N}\ell=v_rM$.
Equality throughout \eqref{eq:raw-maximality} gives
$\abs{P_w(A')}=M$.
Thus $A'$ is another maximizing configuration for the same fixed
parameters $\eta,K,\rho,C,w$.
Step~1 gives (ii), and Steps~2--3 give (i).
This proves (iv).
\end{proof}

\subsection{A common raw point and linear recovery}

Compact generating sets with the same convex hull contain all its extreme points. Applying this observation to a supporting face gives the common raw point we need.

\begin{lemma}[Milman's converse on a supporting face]\label{lem:common-raw-generator}
Let $R_1$ and $R_2$ be nonempty compact subsets of a finite-dimensional real vector space, and suppose that
\[
  \conv R_1=\conv R_2=:K.
\]
For every linear functional $\ell$, there exists a point
\[
  z\in R_1\cap R_2
\]
that maximizes $\ell$ on $K$.
\end{lemma}

\begin{proof}
Choose an extreme point $z$ of the nonempty compact face $F=\{x\in K:\ell(x)=\max_K\ell\}$. Then $z$ is extreme in $K$. Milman's converse to the Krein--Milman theorem \cite[Theorem~3.25]{Rudin1991} gives $\operatorname{ext}(K)\subseteq R_i$ for $i=1,2$, and therefore $z\in R_1\cap R_2$, as required.
\end{proof}

A nonzero leading minor converts equality of normalized Pl\"ucker vectors into a linear factorization. The required row identities follow from Cramer's rule.

\begin{lemma}[Factorization from a common Pl\"ucker vector]\label{lem:anchored-factorization}
Let
\[
  T:(E,p)\to\ell_\infty^N,
  \qquad
  A:(E,q)\to\ell_\infty^N
\]
be contractions. Suppose that, for some signs $\sigma_p,\sigma_q\in\{-1,1\}$,
\begin{equation}\label{eq:common-plucker-point}
  \sigma_p v_p\mathbf m(T)=\sigma_q v_q\mathbf m(A),
\end{equation}
and that the first $n$ rows of $T$ and $A$ form invertible matrices $T_0$ and $A_0$. Define
\[
  L=A_0^{-1}T_0.
\]
Then $L$ is invertible,
\begin{equation}\label{eq:factorization}
  AL=T,
\end{equation}
and
\begin{equation}\label{eq:factorization-volume}
  v_p\abs{\det L}=v_q.
\end{equation}
If, in addition, for some $0<\varepsilon<1$ we have
\[
  \norm{Ay}_\infty>1-\varepsilon
  \qquad(y\in S_q),
\]
then
\begin{equation}\label{eq:anchored-lower-estimate}
  (1-\varepsilon)q(Lx)\leq p(x)
  \qquad(x\in E).
\end{equation}
\end{lemma}

\begin{proof}
Write $a_j$ and $t_j$ for the rows of $A$ and $T$, respectively. The leading blocks satisfy $A_0L=T_0$. For $j>n$ and $1\leq k\leq n$, the minor on the increasing row set $\{1,\ldots,n\}\setminus\{k\}\cup\{j\}$ equals $(-1)^{n-k}\det A_0[k\leftarrow a_j]$, and the same sign occurs for $T$. Dividing this coordinate of \eqref{eq:common-plucker-point} by the nonzero leading coordinate and using Cramer's rule gives
\begin{equation}\label{eq:replacement-ratios}
  (a_jA_0^{-1})_k
  =\frac{\det A_0[k\leftarrow a_j]}{\det A_0}
  =\frac{\det T_0[k\leftarrow t_j]}{\det T_0}
  =(t_jT_0^{-1})_k.
\end{equation}
Thus $a_jA_0^{-1}=t_jT_0^{-1}$, so $a_jL=t_j$. Together with the leading rows, this proves $AL=T$.

The leading coordinate also gives
\[
  \sigma_pv_p\det T_0=\sigma_qv_q\det A_0.
\]
Substitute $T_0=A_0L$, cancel $\det A_0$, and take absolute values to obtain \eqref{eq:factorization-volume}.

Finally, homogeneity extends the lower bound on $S_q$ to $(1-\varepsilon)q(y)\leq\norm{Ay}_\infty$ for every $y\in E$. Therefore
\[
  (1-\varepsilon)q(Lx)
  \leq\norm{ALx}_\infty
  =\norm{Tx}_\infty\leq p(x),
\]
which is \eqref{eq:anchored-lower-estimate}.
\end{proof}

\begin{theorem}[Finite recovery]\label{thm:recovery-certificate}
Let $p$ and $q$ be norms on $E$ with $\cK_p^N=\cK_q^N$ for every $N\geq n$. For every $0<\varepsilon<1$, there exists a bijective linear map $L_\varepsilon:E\to E$ such that
\begin{equation}\label{eq:finite-recovery-norm}
  (1-\varepsilon)q(L_\varepsilon x)\leq p(x)
  \qquad(x\in E)
\end{equation}
and
\begin{equation}\label{eq:finite-recovery-volume}
  v_p\abs{\det L_\varepsilon}=v_q.
\end{equation}
\end{theorem}

\begin{proof}
Apply \Cref{prop:satellite-package} to $q$ and $\varepsilon$, obtaining $N$ and a supporting functional $\ell$. By the equality of the bodies and \Cref{lem:common-raw-generator}, there is a point $z\in\cR_p^N\cap\cR_q^N$ maximizing $\ell$ on the common body. Choose any representations
\[
  z=\sigma_pv_p\mathbf m(T)=\sigma_qv_q\mathbf m(A) \qquad (\sigma_p,\sigma_q\in\{-1,1\}),
\]
with $T$ a $p$-contraction and $A$ a $q$-contraction. Property (iv) of \Cref{prop:satellite-package} ensures that $A$ almost norms $S_q$ and has an invertible leading block. The leading coordinate of $z$ is nonzero, so the leading block of $T$ is invertible as well. The map $L_\varepsilon=A_0^{-1}T_0$ now satisfies both assertions by \Cref{lem:anchored-factorization}.
\end{proof}

\section{Exact recovery and proof of the main theorem}\label{sec:compact-limit}

We now pass from the maps furnished by \Cref{thm:recovery-certificate} to an exact isometry. The argument uses their one-sided norm estimates and exact determinant identities.

\begin{theorem}[Coordinate rigidity]\label{thm:coordinate-rigidity}
Let $p$ and $q$ be norms on $E=\R^n$. If
\[
  \cK_p^N=\cK_q^N
  \qquad\text{for every }N\geq n,
\]
then there exists a surjective linear isometry $L:(E,p)\to(E,q)$.
\end{theorem}

\begin{proof}
For $k\geq0$, set $\varepsilon_k=(k+2)^{-1}$ and choose $L_k=L_{\varepsilon_k}$ from \Cref{thm:recovery-certificate}. Since $\varepsilon_k\leq1/2$,
\[
  \norm{L_k}_{p\to q}\leq(1-\varepsilon_k)^{-1}\leq2.
\]
The operator space is finite-dimensional, so a subsequence converges in operator norm to a linear map $L$. Along this subsequence, for every $x\in E$,
\begin{equation}\label{eq:limit-contraction}
  q(Lx)=\lim_k q(L_kx)
  \leq\lim_k\frac{p(x)}{1-\varepsilon_k}=p(x).
\end{equation}
Continuity of the determinant gives
\begin{equation}\label{eq:limit-volume}
  v_p\abs{\det L}=\lim_k v_p\abs{\det L_k}=v_q>0,
\end{equation}
so $L$ is invertible.

By \eqref{eq:limit-contraction}, $L(B_p)\subseteq B_q$. Linear change of variables \cite[Theorem~3.9]{EvansGariepy2015} and \eqref{eq:limit-volume} yield
\[
  \cL^n(L(B_p))=\abs{\det L}\,v_p=v_q=\cL^n(B_q).
\]
Both sets are compact and convex, and $B_q$ has nonempty interior. Since $B_q=\overline{\operatorname{int}B_q}$, a proper inclusion would make the open set $\operatorname{int}B_q\setminus L(B_p)$ nonempty. That set would contain a ball of positive measure, contradicting the equality of volumes. Hence $L(B_p)=B_q$.

It follows that $\norm{L}_{p\to q}\leq1$ and $\norm{L^{-1}}_{q\to p}\leq1$. Therefore
\[
  p(x)=p(L^{-1}Lx)\leq q(Lx)\leq p(x)
  \qquad(x\in E),
\]
which proves that $L$ is an isometry.
\end{proof}

We return to the original normed spaces to complete the argument.

\begin{proof}[Proof of \Cref{thm:main}]
The zero-dimensional case was settled in \Cref{sec:preliminaries}. In positive dimension, take the coordinate isomorphisms $e_X,e_Y$ and the norms $p,q$ from \eqref{eq:model-norms}. The induced sphere isometry $\delta:S_p\to S_q$ gives equality of the contraction bodies by \Cref{thm:body-equality}. Hence \Cref{thm:coordinate-rigidity} supplies a surjective linear isometry $L:(E,p)\to(E,q)$. The required isometry is
\[
  T=e_Y\circ L\circ e_X^{-1}:X\longrightarrow Y.\qedhere
\]
\end{proof}

\section{Further questions}\label{sec:tingley}

Let $T:X\to Y$ be a linear isometry supplied by \Cref{thm:main}. Then
\[
  T^{-1}\circ\Delta:S_X\longrightarrow S_X
\]
is a self-isometry of the metric unit sphere. Thus, in finite dimensions, the marked extension problem may be separated into two logically distinct parts. The present theorem identifies the ambient space up to linear isometry. What remains is an automorphism problem for a fixed normed space: determine whether every self-isometry of $S_X$ is induced by a linear isometry of $X$. The construction in this paper gives no pointwise relation between $T$ and the original map $\Delta$.

For class-specific extension results and structural viewpoints, see \cite{CabelloSanchez2019,Peralta2018}. Results for self-isometries, such as those for Schreier spaces and their $p$-convexifications \cite{Fakhoury2025}, illustrate the role of additional structure in this remaining problem.

Applied to the underlying real spaces, \Cref{thm:main} also gives real-linear isometric equivalence for finite-dimensional complex normed spaces. A genuinely complex refinement would need to address compatibility with complex scalar multiplication. In infinite dimensions, new substitutes for the Euclidean differentiability, volume, and operator compactness arguments would be required. The metric throughout remains the ambient norm metric from \eqref{eq:ambient-sphere-metric}; intrinsic and geodesic metrics are different invariants.

A quantitative problem is to obtain, in each fixed dimension, an explicit modulus relating the distortion of a sphere map to the Banach--Mazur distance between the ambient spaces, and to determine its dependence on the dimension. Combining object-level rigidity with further information about the self-isometry group of a fixed sphere may also lead to marked extension results.

\section*{Research methods and formalization}\label{sec:formalization}

OpenAI ChatGPT (GPT-5.6 Sol and GPT-6 Astra) was used in developing the proof to generate candidate proof strategies and mathematical arguments, and to assess and refine them. This use was not limited to language editing. The authors critically examined the arguments, discussed their mathematical interpretation, and revised the proofs and exposition. The proofs given here can be checked without access to the AI tools.

A Lean~4 formalization of sphere rigidity, using mathlib, is available in MathlibAnnex, version~0.4.0 \cite{MathlibAnnex040}. It establishes that finite-dimensional real normed spaces with isometric unit spheres are linearly isometric.\footnote{The function-form statement corresponding to \Cref{thm:main} is in \href{https://github.com/r-tanaka-math/mathlib-annex/blob/v0.4.0/MathlibAnnex/Analysis/Normed/Sphere/MetricRigidity.lean}{\texttt{MetricRigidity.lean}} in that release.} An overview of the formal development and its dependencies is provided on the companion project page \cite{SphereRigidityFormalizationProject}.

OpenAI ChatGPT and OpenAI Codex-based tools, both using the GPT-5.6 Sol and GPT-6 Astra models, were used to propose, organize, and revise candidate Lean code and proof obligations in the formalization workflow. The formal result can be checked using the released Lean source. Formal proof checking by Lean is distinct from AI generation of candidate code and from the authors' mathematical review of this manuscript.

\section*{Declaration of generative AI and AI-assisted technologies in the manuscript preparation process}
During the preparation of this work, the authors used OpenAI ChatGPT (GPT-5.6 Sol and GPT-6 Astra) for literature discovery and organization, the development and assessment of proof strategies and mathematical arguments, and drafting and revision of the manuscript. ChatGPT and OpenAI Codex-based tools, using the same models, were also used in the formalization workflow described above. After using these tools, the authors critically reviewed the AI-assisted material and revised the proofs and exposition as needed. The authors take full responsibility for the accuracy, integrity, and final content of this article.


\begin{thebibliography}{99}

\bibitem{BallCurrieOlver1981}
J.~M. Ball, J.~C. Currie, and P.~J. Olver.
\newblock Null Lagrangians, weak continuity, and variational problems of
  arbitrary order.
\newblock {\em Journal of Functional Analysis}, 41(2):135--174, 1981.

\bibitem{Banakh2022}
T.~Banakh.
\newblock Every 2-dimensional {Banach} space has the {Mazur--Ulam} property.
\newblock {\em Linear Algebra and its Applications}, 632:268--280, 2022.

\bibitem{CabelloSanchez2019}
J.~Cabello~S{\'a}nchez.
\newblock A reflection on {Tingley's} problem and some applications.
\newblock {\em Journal of Mathematical Analysis and Applications},
  476(2):319--336, 2019.

\bibitem{CabezasEtAl2022}
D.~Cabezas, M.~Cueto-Avellaneda, D.~Hirota, T.~Miura, and A.~M. Peralta.
\newblock Every commutative {JB}$^*$-triple satisfies the complex
  {Mazur--Ulam} property.
\newblock {\em Annals of Functional Analysis}, 13:60, 2022.

\bibitem{CamposJimenezGarciaPacheco2021}
A.~Campos-Jim{\'e}nez and F.~J. Garc{\'i}a-Pacheco.
\newblock Geometric invariants of surjective isometries between unit spheres.
\newblock {\em Mathematics}, 9(18):2346, 2021.

\bibitem{Conway1990}
J.~B. Conway.
\newblock {\em A Course in Functional Analysis}, volume~96 of {\em Graduate
  Texts in Mathematics}.
\newblock Springer, New York, second edition, 1990.

\bibitem{EvansGariepy2015}
L.~C. Evans and R.~F. Gariepy.
\newblock {\em Measure Theory and Fine Properties of Functions}.
\newblock CRC Press, Boca Raton, revised edition, 2015.

\bibitem{Fakhoury2025}
M.~Fakhoury.
\newblock {Tingley's} problem for {Schreier} spaces and their
  {$p$}-convexifications.
\newblock {\em Journal of Functional Analysis}, 289(10):111122, 2025.

\bibitem{Federer1969}
H.~Federer.
\newblock {\em Geometric Measure Theory}, volume 153 of {\em Grundlehren der
  mathematischen Wissenschaften}.
\newblock Springer, 1969.

\bibitem{Gruber2007}
P.~M. Gruber.
\newblock {\em Convex and Discrete Geometry}, volume 336 of {\em Grundlehren
  der mathematischen Wissenschaften}.
\newblock Springer, Berlin, 2007.

\bibitem{Hatori2022}
O.~Hatori.
\newblock The {Mazur--Ulam} property for uniform algebras.
\newblock {\em Studia Mathematica}, 265(2):227--239, 2022.

\bibitem{HenclMaly2010}
S.~Hencl and J.~Mal{\'y}.
\newblock Jacobians of Sobolev homeomorphisms.
\newblock {\em Calculus of Variations and Partial Differential Equations},
  38:233--242, 2010.

\bibitem{KadetsMartin2012}
V.~Kadets and M.~Mart{\'i}n.
\newblock Extension of isometries between unit spheres of finite-dimensional
  polyhedral {Banach} spaces.
\newblock {\em Journal of Mathematical Analysis and Applications},
  396(2):441--447, 2012.

\bibitem{KupfermanShachar2019}
R.~Kupferman and A.~Shachar.
\newblock A geometric perspective on the {Piola} identity in {Riemannian}
  settings.
\newblock {\em Journal of Geometric Mechanics}, 11(1):59--76, 2019.

\bibitem{Mankiewicz1972}
P.~Mankiewicz.
\newblock On extension of isometries in normed linear spaces.
\newblock {\em Bulletin de l'Acad{\'e}mie Polonaise des Sciences, S{\'e}rie des
  Sciences Math{\'e}matiques, Astronomiques et Physiques}, 20:367--371, 1972.

\bibitem{MathlibAnnex040}
{MathlibAnnex}, version 0.4.0.
\newblock Source release, 2026.
\newblock
  \url{https://github.com/r-tanaka-math/mathlib-annex/releases/tag/v0.4.0}.
  Entry module: \texttt{MathlibAnnex.\allowbreak Projects.\allowbreak
  SphereRigidity}.

\bibitem{MazurUlam1932}
S.~Mazur and S.~Ulam.
\newblock Sur les transformations isom{\'e}triques d'espaces vectoriels
  norm{\'e}s.
\newblock {\em Comptes Rendus de l'Acad{\'e}mie des Sciences de Paris},
  194:946--948, 1932.

\bibitem{McShane1934}
E.~J. McShane.
\newblock Extension of range of functions.
\newblock {\em Bulletin of the American Mathematical Society}, 40(12):837--842,
  1934.

\bibitem{MoriOzawa2020}
M.~Mori and N.~Ozawa.
\newblock {Mankiewicz's} theorem and the {Mazur--Ulam} property for
  {$C^*$}-algebras.
\newblock {\em Studia Mathematica}, 250(3):265--281, 2020.

\bibitem{Peralta2018}
A.~M. Peralta.
\newblock A survey on {Tingley's} problem for operator algebras.
\newblock {\em Acta Scientiarum Mathematicarum}, 84:81--123, 2018.

\bibitem{PeraltaSvarc2025}
A.~M. Peralta and R.~\v{S}varc.
\newblock A strengthened {Kadison's} transitivity theorem for unital
  {JB}$^*$-algebras with applications to the {Mazur--Ulam} property.
\newblock {\em Analysis and Mathematical Physics}, 15:91, 2025.

\bibitem{Rademacher1919}
H.~Rademacher.
\newblock {\"U}ber partielle und totale Differenzierbarkeit von Funktionen
  mehrerer Variabeln und {\"u}ber die Transformation der Doppelintegrale.
\newblock {\em Mathematische Annalen}, 79:340--359, 1919.

\bibitem{Rudin1991}
W.~Rudin.
\newblock {\em Functional Analysis}.
\newblock McGraw-Hill, New York, second edition, 1991.

\bibitem{SphereRigidityBrief2026}
{Sphere Rigidity}.
\newblock Exact Mathematics with AI: Brief Report, Draft R2, 2026.
\newblock
  \url{https://exactmathematics.org/reports/sphere-rigidity/draft-r2/sphere-rigidity-brief-report.pdf}.

\bibitem{SphereRigidityFormalizationProject}
{Sphere Rigidity}: Companion project page.
\newblock Exact Mathematics with AI.
\newblock
  \url{https://exactmathematics.org/mathlibannex/projects/sphere-rigidity/}.

\bibitem{Tanaka2014}
R.~Tanaka.
\newblock A further property of spherical isometries.
\newblock {\em Bulletin of the Australian Mathematical Society},
  90(2):304--310, 2014.

\bibitem{Tanaka2017}
R.~Tanaka.
\newblock Spherical isometries of finite-dimensional {$C^*$}-algebras.
\newblock {\em Journal of Mathematical Analysis and Applications},
  445(1):337--341, 2017.

\bibitem{Taylor1947}
A.~E. Taylor.
\newblock A geometric theorem and its application to biorthogonal systems.
\newblock {\em Bulletin of the American Mathematical Society},
  53(6):614--616, 1947.

\bibitem{Tingley1987}
D.~Tingley.
\newblock Isometries of the unit sphere.
\newblock {\em Geometriae Dedicata}, 22:371--378, 1987.

\end{thebibliography}
\end{document}